\documentclass[10pt]{article}
\font\smallit=cmti10

\usepackage{amssymb,latexsym,amsmath,epsfig,amsthm} 
\usepackage{amsthm}
\usepackage{fancyhdr}

\makeatletter

\renewcommand\section{\@startsection {section}{1}{\z@}
 {-30pt \@plus -1ex \@minus -.2ex}
 {2.3ex \@plus.2ex}
 {\normalfont\normalsize\bfseries}}

\renewcommand\subsection{\@startsection{subsection}{2}{\z@}
 {-3.25ex\@plus -1ex \@minus -.2ex}
 {1.5ex \@plus .2ex}
 {\normalfont\normalsize\bfseries}}

\renewcommand{\@seccntformat}[1]{\csname the#1\endcsname. }

\makeatother

\newtheorem{theorem}{Theorem}
\newtheorem{lemma}{Lemma}

\newtheorem{example}{Example}
\newtheorem*{remark}{Remark}
\numberwithin{equation}{section}

\begin{document}


\begin{center}
\uppercase{\bf On some generalizations of mean value theorems for arithmetic functions of two variables}
\vskip 20pt
{\bf Noboru Ushiroya}\\
{\smallit National Institute of Technology, Wakayama College, \\ Gobo, Wakayama, Japan}\\
{\tt ushiroya@wakayama-nct.ac.jp}\\ 
\vskip 10pt
\end{center}
\vskip 30pt
\vskip 30pt


\centerline{\bf Abstract}

\noindent
Let $f: \mathbb{N}^2 \mapsto \mathbb{C}$  be an arithmetic function of two variables. 
We study the existence of the limit: 
  \[\displaystyle  \lim_{x  \to  \infty}  \frac{1}{x^2 (\log x)^{k-1}} \sum_{n_1 ,  n_2 \le x} f  (n_1,  n_2)  \]
where $k$ is a fixed positive integer. Moreover, we express this limit as an infinite product over all prime numbers in the case that $f$ is a multiplicative function of two variables. This study is a generalization of Cohen-van der Corput's results to the case of two variables.

\pagestyle{myheadings} 
 \thispagestyle{empty} 
 \baselineskip=12.875pt 
 \vskip 30pt 


\section{Introduction}
Let $\mu$ denote the $\mathrm{the \ M \ddot{o} bius}$ function and let $\mu_k=\underbrace{\mu * \mu * \cdots * \mu}_{k}$ be the $k-$folded Dirichlet convolution of $\mu$, that is, $ \mu_k (n)=\sum_{d_1 d_2 \cdots d_k=n} \mu(d_1) \mu(d_2) \dots \mu(d_k)$ for every $n$. Cohen \cite{co2} proved that if $f:\mathbb{N} \mapsto \mathbb{C}$ is an arithmetic function satisfying $\sum_{n=1}^\infty |(f* \mu_k)(n)| / n < \infty$, then 
\begin{equation}
\displaystyle  \lim_{ x \to \infty} \frac{1}{x (\log x)^{k-1} } \sum_{n \le x} f(n)=\frac{1}{(k-1)!} \sum_{n=1}^\infty \frac{(f* \mu_k)(n)}{n} . 
\end{equation}

Van der Corput \cite{van} proved that if $f:\mathbb{N} \mapsto \mathbb{C}$ is a multiplicative function satisfying $ \prod_{p \in  \mathcal{P}} ( \sum_{\nu=0}^\infty |(f* \mu_k)(p^\nu)| / p^\nu )< \infty$ where $\mathcal{P}$ is the set of prime numbers, then
\begin{equation}
\lim_{ x \to \infty} \frac{1}{x (\log x)^{k-1} } \sum_{n \leq x} f(n)=\frac{1}{(k-1)!} \prod_{p \in \mathcal{P}} \Bigl(1-\frac{1}{p}\Bigr)^k \Bigl(\sum_{\nu=0}^{\infty} \frac{f(p^\nu)}{p^\nu} \Bigr).   
\end{equation}
We would like to generalize these results to the case in which $f$ is an arithmetic function of two variables and obtain several interesting examples. 

Let $\gcd (n_1,n_2)$ denote the greatest common divisor of $n_1$ and $n_2$, $ \sigma (n)$ the sum of divisors of $n$, and $\varphi (n)$ Euler's totient function. Cohen \cite{co3} proved that
\begin{equation}
 \sum_{n_1,n_2 \le x} \sigma (\gcd (n_1,n_2))= x^2 \Bigl(\log x +2 \gamma -\frac{1}{2}-\frac{\zeta(2)}{2} \Bigr)+O(x^{\frac{3}{2}} \log x), 
\end{equation}
\begin{equation}
 \sum_{n_1,n_2 \le x} \varphi (\gcd (n_1,n_2))= \frac{x^2}{\zeta^2(2)} \Bigl(\log x +2 \gamma -\frac{1}{2}-\frac{\zeta(2)}{2}-\frac{2 \zeta'(2)}{\zeta(2)} \Bigr)+O(x^{\frac{3}{2}} \log x),
\end{equation}
where $\zeta(n)$ is the Riemann zeta function. 

Next we consider two functions $s$ and $c$, where $ s(n_1,n_2)= \sum_{d_1|n_1,d_2|n_2} \mathrm{gcd} (d_1,d_2)$ and $ c(n_1,n_2)= \sum_{d_1|n_1,d_2|n_2} \varphi(\mathrm{gcd} (d_1,d_2))$. Nowak and T$\acute{\rm{o}}$th \cite{no} proved that
\begin{equation}
\sum_{n_1,n_2 \le x} s  (n_1,n_2)= \frac{2}{\pi^2} x^2 (\log^3 x +a_1 \log^2 x+a_2 \log x+a_3)+(x^{\frac{1117}{701}+\varepsilon}) , 
\end{equation}
\begin{equation}
\sum_{n_1,n_2 \le x} c (n_1,n_2)= \frac{12}{\pi^4} x^2 (\log^3 x +b_1 \log^2 x+b_2 \log x+b_3)+(x^{\frac{1117}{701}+\varepsilon}), 
\end{equation}
where $a_1,a_2,a_3,b_1,b_2,b_3$ are explicit constants. 

We would like to obtain these leading coefficients in (1.3) $\sim$ (1.6) by a systematic method. We will calculate those leading coefficients in Example 3, 4, 7 and 8 in Section 5. Although we cannot obtain remainder terms by our theorems, our method for obtaining leading terms is very simple and is applicable to many arithmetic functions of two variables.

\section{Some Results}
Let $\tilde{\mu} (n_1,n_2)$ denote the Dirichlet inverse of the $\mathrm{gcd}$ function, that is, $\tilde{\mu}$ is the function which satisfies $(\tilde{\mu} * \mathrm{gcd})(n_1,n_2)= \delta (n_1,n_2)$ for every $n_1,n_2 \in \mathbb{N}$, where $\delta(n_1,n_2)=1$ or $0$ according to whether $n_1=n_2=1$ or not.
Let $x \wedge y$ denote $\min(x,y)$. We first establish the following theorem.

\begin{theorem} Let $f$ be an arithmetic function of two variables satisfying
\begin{equation} \displaystyle \sum_{n_1,n_2=1}^{\infty} \frac{|(f*\tilde{\mu})(n_1,n_2)|}{n_1 n_2}< \infty . \end{equation} 
Then we have 
\begin{equation}
 \lim_{x,y \to \infty} \frac{1}{xy  \log x \wedge y} \sum_{n_1 \le x ,  n_2 \le y} f (n_1,n_2) =  \frac{1}{\zeta (2)} \sum_{n_1,  n_2=1}^{\infty} \frac{(f*\tilde{\mu})(n_1,n_2)}{n_1 n_2} . \label{eq:th1}
\end{equation}
\end{theorem}

The proof of Theorem 1 will be given in the next section. To proceed to the next theorem, we need some notations. Let  
\[\tau_k(n_1,n_2)=(\underbrace{\bold{1} * \bold{1} * \cdots * \bold{1}}_{k}) (n_1,n_2) \] 
stand for the $k-$folded Dirichlet convolution of the function $\bold{1}$, where $\bold{1}(n_1,n_2)=1$ for every $n_1,n_2 \in \mathbb{N}$. Let $\mu_k=\tau_k^{-1}$ denote the Dirichlet inverse of $\tau_k$. 
Note that $\mu_1(n_1,n_2)=\mu (n_1) \mu (n_2)$. Similarly, let 
\begin{align*}
\tilde{\tau}_1 (n_1,n_2)&= \mathrm{gcd} (n_1,n_2) ,  \\
\tilde{\tau}_k(n_1,n_2) &= (\underbrace{\bold{1} * \bold{1} * \cdots * \bold{1}}_{k-1} * \gcd ) (n_1,n_2)  \quad \mathrm{if} \quad k \ge 2 . 
\end{align*}
We also denote $\tilde{\mu}_k =\tilde{\tau}_k^{-1}$ the Dirichlet inverse of $\tilde{\tau}_k$. Note that $\tilde{\mu}_1=\tilde{\mu}=\gcd^{-1} $ and $\tilde{\mu}_k=\mu_{k-1} * \tilde{\mu}$ \ if \ $k \ge 2$. The next theorem is an extension of Cohen's theorem (1.1) to the case in which $f$ is an arithmetic function of two variables.

\begin{theorem} Let $f$ be an arithmetic function of two variables and let $k \in \mathbb{N}$. \\
{\rm{(i)}} Suppose
\begin{equation}
\displaystyle \sum_{n_1,n_2=1}^{\infty} \frac{|(f* \mu_k)(n_1,n_2)|}{n_1 n_2}< \infty . \label{eq:th2(i)1}
\end{equation}
Then we have
\begin{equation}
\displaystyle \lim_{x,y \to \infty} \frac{1}{xy (\log x  \log y)^{k-1} } \sum_{n_1 \le x ,  n_2 \le y} f (n_1,n_2) = C_k  \sum_{n_1,  n_2=1}^{\infty} \frac{(f*\mu_k)(n_1,n_2)}{n_1 n_2} ,
\label{eq:th2i}
\end{equation}
where $ \displaystyle C_k=\frac{1}{((k-1) !)^2} . $ \\
\\
{\rm{(ii)}} Suppose 
\begin{equation}
\displaystyle \sum_{n_1,n_2=1}^{\infty} \frac{|(f* \tilde{\mu}_k)(n_1,n_2)|}{n_1 n_2}< \infty . \label{eq:th2(ii)1}
\end{equation}
Then we have
\begin{equation}
\displaystyle \lim_{x \to \infty} \frac{1}{x^2 (\log x)^{2k-1} } \sum_{n_1 , n_2 \le x } f (n_1,n_2) = \tilde{C}_k  \sum_{n_1,  n_2=1}^{\infty} \frac{(f* \tilde{\mu}_k)(n_1,n_2)}{n_1 n_2}  ,
\label{eq:th2ii}
\end{equation}
where $ \displaystyle \tilde{C}_k=\frac{1}{\zeta(2)} \frac{1}{((k-1) !)^2 (2k-1)} . $ 
\end{theorem}

\begin{remark} \rm{In part} $\mathrm{(ii)}$, we do not deal with: \\
$ \lim_{x,y \to \infty} (xy (\log x  \log y)^{k-1} \log x \wedge y )^{-1} \sum_{n_1 \le x ,  n_2 \le y} f (n_1,n_2)$  since it is too complicated and we cannot obtain a simple formula.
\end{remark}

The proof of Theorem 2 will also be given in the next section. 

\section{Proof of Theorem 1 and Theorem 2}
The following lemma is well known (cf. Cohen \cite{co2}) and will be needed later.

\begin{lemma} For fixed $\alpha \ge 0$ and all $x$, we have
\begin{equation}
 \sum_{n \le x}  \frac{\log^{\alpha} n}{n}= \frac{\log^{\alpha+1} x}{\alpha+1} +O(1).
\end{equation}
\end{lemma}

It is also well known that $\sum_{n_1,n_2 \le x }  \gcd (n_1,n_2) = x^2 \log x / \zeta(2) + cx^2 +o(x^2),$
where $c$ is a suitable constant (cf. Ces$\grave{\rm{a}}$ro \cite{ce}). We would like to modify this formula as follows.

\begin{lemma}
\begin{equation}
\lim_{x,  y \to \infty} \frac{1}{xy \log x \wedge y} \sum_{n_1 \le x,  n_2 \le y} \gcd ( n_1 ,  n_2 ) =\frac{1}{\zeta (2)}.  \label{eq:lem2}
\end{equation}
\end{lemma}
\begin{proof} Let 
\begin{align*}
A(x,y)&= \# \{ (n_1,  n_2) : \ 1 \le n_1 \le x, \  1 \le n_2 \le y, \  \gcd (n_1,n_2)=1 \}   \\
 &= \sum_{n_1 \le x,  n_2 \le y} \mu^2 ((\mathrm{gcd}(n_1, n_2))^2) . 
\end{align*}
Applying Theorem 7 in Ushiroya \cite{u2} to the function $\mu^2 ((\mathrm{gcd}(n_1, n_2))^2)$ we have 
\[ \lim_{x,  y \to \infty} \frac{1}{xy} A(x,y)=\frac{1}{\zeta (2)}. \]
From this we have
\[ \sum_{n_1 \le x, n_2 \le y} \gcd (n_1, n_2) \]
\vspace{-0.4cm}
\begin{align*}
&= \sum_{1 \le d \le x \wedge y} d \ \# \{ (n_1,  n_2) ; \  1 \le n_1 \le x, \ 1 \le n_2 \le y, \  \gcd(n_1,n_2)=d \}  \\
&= \sum_{1 \le d \le x \wedge y} d \ \# \{ (n_1^{'}, \ n_2^{'}) ; \ \ 1 \le n_1^{'} \le \frac{x}{d}, \ \ 1 \le n_2^{'} \le \frac{y}{d}, \ \ \gcd(n_1^{'},n_2^{'})=1 \}   \\
&= \sum_{1 \le d \le x \wedge y} d A(\frac{x}{d}, \ \frac{y}{d}) =\sum_{1 \le d \le x \wedge y} d \Bigl( \frac{1}{\zeta(2)} \frac{x}{d} \frac{y}{d} +o(\frac{x}{d} \frac{y}{d}) \Bigr)   \\
&= \frac{1}{\zeta(2)} xy \log x \wedge y   +o(xy \log x \wedge y) , 
\end{align*}
which implies (\ref{eq:lem2}).
\end{proof}

\begin{lemma} Let $a(n_1,n_2)$ be an arithmetic function of two variables satisfying $ \sum_{n_1,n_2=1}^\infty |a(n_1,n_2)| < \infty $. Then we have
\vspace{-0.3cm}
\begin{equation}
\lim_{x,  y \to \infty} \frac{1}{\log x \wedge y} \sum_{n_1 \le x ,   n_2 \le y} a(n_1,n_2) \log \frac{x}{n_1} \wedge \frac{y}{n_2} =\sum_{n_1,n_2=1}^\infty a(n_1,n_2). 
\label{eq:lem3}
\end{equation}
\end{lemma}
\begin{proof} We put $M=\sum_{n_1,n_2=1}^\infty a(n_1,n_2)$. Then for any $\varepsilon>0$, there exists $N>0$ such that $\Bigl|\sum_{n_1, n_2<N}  a(n_1,n_2)-M \Bigr|< \varepsilon $. If we take $x$ and $y$ sufficiently large such that $x \wedge y > N$, then we have 
\[\sum_{n_1 \le x ,  n_2 \le y} a(n_1,n_2) \log \frac{x}{n_1} \wedge \frac{y}{n_2} = \sum_{n_1,n_2 <N } a(n_1,n_2) \Bigl( \log \frac{x}{n_1} \wedge \frac{y}{n_2}-\log x \wedge y \Bigr)  \]
\[ +\log x \wedge y \sum_{n_1 ,  n_2 <N} a(n_1,n_2)  + \sum_{\substack{n_1 \le x, n_2 \le y \\ \ n_1 \wedge n_2 \ge N}} a(n_1,n_2) \log \frac{x}{n_1} \wedge \frac{y}{n_2} =:I_1+I_2+I_3, \]
where
\begin{align*}
& I_1  \ll \Bigl(\sup_{n_1,n_2 <N} \log \Bigl| \frac{\frac{x}{n_1} \wedge \frac{y}{n_2}}{ x \wedge y}  \Bigr| \Bigr) \sum_{n_1,n_2 <N } a(n_1,n_2) \ll \log N, \hspace{3cm} \\
& |I_2 -M \log x \wedge y | < \varepsilon  \log x \wedge y,  \\
{\rm{and}} \qquad & \\
& I_3 \ll \log x \wedge y \sum_{\substack{n_1 \le x, n_2 \le y \\ \ n_1 \wedge n_2 \ge N}}  a(n_1,n_2) \ll  \varepsilon \log x \wedge y . 
\end{align*}
Therefore we have
\[\limsup_{ x,y \to \infty} \Bigl| \frac{1}{\log x \wedge y}  \sum_{n_1 \le x , \ n_2 \le y} a(n_1,n_2) \log \frac{x}{n_1} \wedge \frac{y}{n_2} -M   \Bigr| \ll 2\varepsilon. \]
Since $\varepsilon$ is arbitrary, (\ref{eq:lem3}) holds.
\end{proof}

Now we can prove Theorem 1. 

\begin{proof}[Proof of Theorem 1]
We put $g=f* \tilde{\mu}$. Noting that $\tilde{\mu} *\tilde{\tau}_1=\delta$ we have
\[ \sum_{n_1 \le x  ,  n_2 \le y} f (n_1,n_2)=\sum_{n_1 \le x  ,  n_2 \le y} (f*\tilde{\mu} *\tilde{\tau}_1) (n_1,n_2)=\sum_{n_1 \le x, n_2 \le y} (g *\tilde{\tau}_1) (n_1,n_2)                 \] 
\[ =\sum_{d_1 \delta_1 \le x  , \ d_2 \delta_2 \le y} g(d_1 , d_2) \tilde{\tau} (\delta_1  ,  \delta_2)=\sum_{n_1  \le x ,  n_2 \le y} g(n_1 \ , n_2)  \sum_{ \delta_1 \le \frac{x}{n_1}  ,   \delta_2 \le \frac{y}{n_2}}  \tilde{\tau}_1 (\delta_1  ,  \delta_2).                 \] 

From Lemma 2 we see that this equals 
\[ \sum_{n_1 \le x , n_2 \le y} g(n_1  , n_2) \Bigl\{\frac{1}{\zeta (2)} \frac{x}{n_1} \frac{y}{n_2} \log (\frac{x}{n_1} \wedge \frac{y}{n_2} ) + o \bigl(\frac{x}{n_1} \frac{y}{n_2} \log (\frac{x}{n_1} \wedge \frac{y}{n_2} ) \bigr) \Bigr\} .                 \] 

Applying Lemma 3 to the function $a(n_1,n_2)=g(n_1,n_2) / n_1 n_2$, we see that the above equals 
\[  \frac{xy \log x \wedge y}{\zeta(2)} \sum_{n_1,n_2=1}^\infty \frac{g(n_1,n_2)}{n_1 n_2} + o(xy \log x \wedge y), \]
which implies (\ref{eq:th1}). Thus the proof of Theorem 1 is now complete.
\end{proof}

Next we prove several lemmas needed later.

\begin{lemma} 
$\displaystyle \sum_{n_1 \le x , n_2 \le y}  \log n_1 \wedge n_2 = xy \log x \wedge y+o(xy \log x \wedge y). $
\end{lemma}
\begin{proof} Without loss of generality, we may assume that $y \le x.$ Let $[x]$ denote the greatest integer that is less than or equal to $x$.
Using the well known formula $\sum_{1 \le n \le x} \log n=x \log x -x + O(\log x), $ we have
\vspace{-0.2cm}
\begin{align*}
&\displaystyle \sum_{n_1 \le x  ,  n_2 \le y} \log n_1 \wedge n_2 =\sum_{ n_2 \le y} \Bigl(\sum_{ n_1=1}^{ n_2}  \log n_1+\sum_{ n_1= n_2+1}^{[x]}  \log n_2 \Bigr)  \\
&= \sum_{ n_2 \le y} \Bigl(n_2 \log n_2-n_2+O(\log n_2)+([x]-n_2)  \log n_2 \Bigr)   \\
&= \sum_{ n_2 \le y} \Bigl([x] \log n_2-n_2+O(\log n_2) \Bigr)=[x](y \log y-y+O(\log y)) +O(y^2)   \\
&= xy \log x \wedge y+o(xy \log x \wedge y).  
\end{align*}
\end{proof}

\begin{lemma}
$\displaystyle \sum_{n_1,n_2 \le x}  \frac{\log n_1 \wedge n_2}{n_1 n_2} = \frac{1}{3} (\log x)^3+o((\log x)^3) . $
\end{lemma}
\begin{proof}  Using Lemma 1 we have
\[ \sum_{n_1,n_2 \le x}  \frac{\log n_1 \wedge n_2}{n_1 n_2} = \sum_{n_2 \le x} \Bigl( \sum_{n_1 \le n_2} \frac{\log n_1}{n_1 n_2} +\sum_{n_2 < n_1 \le x} \frac{\log n_2}{n_1 n_2} \Bigr)  \]
\begin{align*}
&=\sum_{n_2 \le x} \Bigl( \frac{(\log n_2)^2 + O(1)  }{2 n_2}+ \frac{\log n_2(\log x -\log n_2+O(1))  }{n_2} \Bigr)    \\
&= \frac{1}{6} (\log x)^3+\frac{1}{2} (\log x)^3-\frac{1}{3} (\log x)^3+o((\log x)^3)= \frac{1}{3} (\log x)^3+o((\log x)^3).    
\end{align*}
\end{proof}

\begin{lemma} For fixed $\alpha, \beta \ge 0$ and all $x$, we have
\begin{align}
\sum_{n_1 ,n_2 \le x} \frac{ (\log n_1)^\alpha  (\log n_2)^\beta  \ \log \frac{x}{n_1} \wedge \frac{x}{n_2}}{n_1 n_2} = & \frac{ (\log x)^{\alpha+\beta+3} }{(\alpha+1)(\beta+1) (\alpha+\beta+3)}  \nonumber \\
& +o \bigl((\log x)^{\alpha+\beta+3} \bigr) . \label{eq:lem6}
\end{align}
\end{lemma}
\begin{proof} Using Lemma 1 we see that the left side of (\ref{eq:lem6}) equals
\[ \sum_{n_2 \le x} \Bigl( \sum_{n_1 \le n_2} \frac{(\log n_1)^\alpha  (\log n_2)^\beta  \log \frac{x}{n_2}}{n_1 n_2} +\sum_{n_2 < n_1 \le x} \frac{(\log n_1)^\alpha  (\log n_2)^\beta  \log \frac{x}{n_1}}{n_1 n_2} \ \Bigr)     \]
\begin{align*}
& = \sum_{n_2 \le x} \Bigl( \sum_{n_1 \le n_2} \frac{(\log n_1)^\alpha  (\log n_2)^\beta  (\log x -\log n_2)}{n_1 n_2} \\
& \qquad \qquad +\sum_{n_2 < n_1 \le x} \frac{(\log n_1)^\alpha  (\log n_2)^\beta  (\log x -\log n_1)}{n_1 n_2}  \Bigr)     \\
& =  \sum_{n_2 \le x} \Bigl(  \frac{((\log n_2)^{\alpha+1}+O(1))(\log n_2)^{\beta} (\log x-\log n_2)}{(\alpha+1) n_2}  \\
& \qquad \qquad + \frac{( (\log x)^{\alpha+1} -(\log n_2)^{\alpha+1}+O(1) ) (\log n_2)^\beta  \log x }{(\alpha+1) n_2}  \\
& \qquad \qquad - \frac{( (\log x)^{\alpha+2} -(\log n_2)^{\alpha+2}+O(1) ) (\log n_2)^\beta }{(\alpha+2) n_2} \   \Bigr)  \\ 
& = \sum_{n_2 \le x} \Bigl(\frac{1}{\alpha+1}-\frac{1}{\alpha+2} \Bigr) \frac{(\log x)^{\alpha+2} (\log n_2)^{\beta} -(\log n_2)^{\alpha+\beta+2}}{n_2}+o \bigl((\log x)^{\alpha+\beta+3} \bigr) \\
& = \frac{1}{(\alpha+1)(\alpha+2)} \Bigl(\frac{(\log x)^{\alpha+\beta+3}}{\beta+1}-\frac{(\log x)^{\alpha+\beta+3}}{\alpha+\beta+3} \Bigr)+o \bigl((\log x)^{\alpha+\beta+3} \bigr) \\
& = \frac{ (\log x)^{\alpha+\beta+3} }{(\alpha+1)(\beta+1) (\alpha+\beta+3)}  +o \bigl((\log x)^{\alpha+\beta+3} \bigr) .
\end{align*}
This proves Lemma 6.
\end{proof}

The next lemma gives a partial summation formula in the case of a function of two variables.
\begin{lemma} Let $a(n_1,n_2)$ be an arithmetic function of two variables and \\ let $M(x,y)=\sum_{n_1 \le x , n_2 \le y} a(n_1,n_2)$. Then we have
\[ \sum_{n_1 \le x ,  n_2 \le y} \frac{a(n_1,n_2)}{n_1 n_2}=\sum_{\substack{n_1 \le x \\  n_2 \le y}} \frac{M(n_1,n_2)}{n_1(n_1+1) n_2(n_2+1)} +\sum_{n_1 \le x} \frac{M(n_1,y)}{n_1(n_1+1) ([y]+1)}  \]
\begin{equation}
 +\sum_{n_2 \le y} \frac{M(x,n_2)}{n_2(n_2+1) ([x]+1)} + \frac{M(x,y)}{([x]+1)([y]+1)},  \label{eq:lem7}
\end{equation}
where $[x]$ is the greatest integer that is less than or equal to $x$.
\end{lemma}
\begin{proof} We put $M(x.y)=0$ if $x<1$ or $y<1$ for convenience. Then we see that the left side of (\ref{eq:lem7}) equals
\[  \sum_{n_1 \le x ,  n_2 \le y} \frac{M(n_1,n_2)-M(n_1-1,n_2)-M(n_1,n_2-1)+M(n_1-1,n_2-1)}{n_1 n_2}  \]
\[=\sum_{n_1 \le x ,  n_2 \le y} M(n_1,n_2) \Bigl\{\frac{1}{n_1 n_2} -\frac{1}{(n_1+1) n_2} -\frac{1}{n_1 (n_2+1)} +\frac{1}{(n_1+1)( n_2+1)}  \Bigr\}  \]
\[+ \sum_{n_2 \le y} \frac{M(x,n_2)}{([x]+1) n_2}+ \sum_{n_1 \le x} \frac{M(n_1,y)}{n_1([y]+1)} -\sum_{n_1 \le x} \frac{M(n_1,y)}{(n_1+1)([y]+1)}  \]
\[- \sum_{n_2 \le y} \frac{M(x,n_2)}{([x]+1) (n_2+1)}+ \frac{M(x,y)}{([x+1]+1)([y]+1)}   \]
\[=\sum_{n_1 \le x ,  n_2 \le y} M(n_1,n_2) \frac{1}{n_1(n_1+1) n_2(n_2+1)} +\sum_{n_2 \le y} \frac{M(x,n_2)}{([x]+1) }  \Bigl(\frac{1}{n_2}-\frac{1}{n_2+1} \Bigr)  \]
\[ +\sum_{n_1 \le x} \frac{M(n_1,y)}{([y]+1) }  \Bigl(\frac{1}{n_1}-\frac{1}{n_1+1} \Bigr) + \frac{M(x,y)}{([x]+1)([y]+1)},   \]
which equals the right side of (\ref{eq:lem7}).
\end{proof}

The next lemma is an extension of Proposition 5 in van der Corput \cite{van} to the case of arithmetical functions of two variables.

\begin{lemma} 
Let $a, b$ be arithmetical functions of two variables and let $c=a*b$. \\
For $\alpha, \beta \ge 0$, we assume that
\[  \lim_{x,y \to \infty} \frac{1}{xy (\log x)^\alpha (\log y)^{\beta}} \sum_{n_1 \le x  , n_2 \le y} a(n_1, n_2)=A, \]
where $A$ is a constant. 
\[ \mathrm{(i)} \ If \lim_{x,y \to \infty} \frac{1}{xy } \sum_{n_1 \le x  ,  n_2 \le y} b(n_1, n_2)= B,  \ \ where \  B \ is \ a  \ constant, \ then  \hspace{5in} \]
\begin{equation}\label{eq:lem8i}
\lim_{x \to \infty} \frac{1}{xy (\log x)^{\alpha+1} (\log y)^{\beta+1}} \sum_{n_1 \le x , n_2 \le y} c(n_1, n_2)= \frac{AB}{(\alpha+1)(\beta+1) } . 
\end{equation}
\[\mathrm{(ii)} \ If \lim_{x,y \to \infty} \frac{1}{xy \log x \wedge y} \sum_{n_1 \le x  ,  n_2 \le y} b(n_1, n_2)= B, \ \  where \  B \ is \ a \ constant, \ then  \hspace{5in}  \]
\begin{equation}\label{eq:lem8ii}
\lim_{x \to \infty} \frac{1}{x^2 (\log x)^{\alpha+\beta+3}  } \sum_{n_1,  n_2 \le x} c(n_1, n_2)= \frac{AB}{(\alpha+1)(\beta+1) (\alpha+\beta+3)} . 
\end{equation}
\end{lemma}
\begin{proof} We first prove (i). We have
\[ \sum_{n_1 \le x,  n_2 \le y} c(n_1,n_2) =\sum_{n_1 \le x,  n_2 \le y} (a*b)(n_1,n_2) =\sum_{ \ell_1 m_1  \le x ,  \  \ell_2 m_2 \le y}  a(\ell_1,\ell_2) b(m_1,m_2)  \] 
\[= \sum_{ \ell_1 m_1 \le x,  \ \ell_2 m_2 \le y} a(\ell_1,\ell_2) \bigl( b(m_1,m_2) -B \bigr) \]
\[+ B \sum_{ \ell_1 m_1 \le x, \ \ell_2 m_2 \le y}  \ \Bigl(a(\ell_1,\ell_2)-A(\log \ell_1)^\alpha (\log \ell_2)^{ \beta} \Bigr)         \]  
\[+ AB \sum_{\ell_1 m_1 \le x, \  \ell_2 m_2 \le y}   (\log \ell_1)^\alpha (\log \ell_2)^{ \beta} =: I_1+I_2+I_3,       \] 
where, by Lemma 7 and Lemma 1, 
\begin{align*}
I_1 & = \sum_{\ell_1  \le x,   \ell_2 \le y} a(\ell_1,\ell_2) \sum_{\substack{m_1  \le x /\ell_1 \\   m_2 \le y / \ell_2} } \bigl( b(m_1,m_2) -B  \bigr) = \sum_{\ell_1  \le x,   \ell_2 \le y} a(\ell_1,\ell_2)  \ o \Bigl( \frac{xy}{\ell_1 \ell_2}  \Bigr)  \\
& =o\Bigl(xy \sum_{\ell_1  \le x,   \ell_2 \le y} \frac{A \ell_1 \ell_2 (\log \ell_1)^\alpha (\log \ell_2)^{\beta}}{\ell_1 (\ell_1+1) \ell_2 (\ell_2+1)} \Bigr)= o(xy (\log x)^{\alpha+1} (\log y)^{\beta+1}), \\
I_2 &= B \sum_{\ell_1  \le x,   \ell_2  \le y}  \bigl(a(\ell_1,\ell_2)-A(\log \ell_1)^\alpha (\log \ell_2)^{ \beta} \bigr) \sum_{m_1  \le x / \ell_1 ,   m_2 \le y / \ell_2 } 1  \\
&= B \sum_{\ell_1  \le x,   \ell_2  \le y}   \frac{xy}{\ell_1 \ell_2} \bigl(a(\ell_1,\ell_2)-A(\log \ell_1)^\alpha (\log \ell_2)^{ \beta} \bigr)  \\
&=B \sum_{\ell_1  \le x,  \ell_2  \le y}  xy \ \frac{ o \bigl( \ell_1 \ell_2 (\log \ell_1)^\alpha (\log \ell_2)^{\beta} \bigr)}{\ell_1 (\ell_1+1) \ell_2 (\ell_2+1)}   = o \bigl(xy (\log x)^{\alpha+1} (\log y)^{\beta+1} \bigr), \\
{\rm{and}} \qquad & \\
I_3 &=AB \sum_{\ell_1 \le x, \ell_2  \le y}   (\log \ell_1)^\alpha (\log \ell_2)^{ \beta} \sum_{m_1  \le x /\ell_1 ,   m_2 \le y / \ell_2 } 1  \\
&=AB \sum_{\ell_1 \le x, \ell_2  \le y}   \frac{xy}{\ell_1 \ell_2} (\log \ell_1)^\alpha (\log \ell_2)^{ \beta}  \\
&= \frac{AB}{(\alpha+1)(\beta+1)} xy (\log x )^{\alpha+1} (\log y)^{\beta+1}+o \bigl(xy (\log x)^{\alpha+1} (\log y)^{\beta+1} \bigr). \end{align*}
Therefore (\ref{eq:lem8i}) holds. This proves (i).

Next we prove (ii). Similarly we have
\[ \sum_{n_1, \ n_2 \le x} c(n_1,n_2) =\sum_{\ell_1 m_1 \le x,  \  \ell_2 m_2 \le x} a(\ell_1,\ell_2) \ b(m_1,m_2)   \] 
\[= \sum_{\ell_1 m_1 \le x,   \ell_2 m_2 \le x} a(\ell_1,\ell_2) \bigl( b(m_1,m_2) -B \log m_1 \wedge m_2  \bigr)       \hspace{3cm}         \]  
\[+ B \sum_{ \ell_1 m_1 \le x,  \ell_2 m_2 \le x}  \bigl(a(\ell_1,\ell_2)-A(\log \ell_1)^\alpha (\log \ell_2)^{\beta} \bigr) \log m_1 \wedge m_2                 \] 
\[+ AB \sum_{ \ell_1 m_1 \le x,  \ell_2 m_2 \le x}   (\log \ell_1)^\alpha (\log \ell_2)^{\beta} \log m_1 \wedge m_2 =:J_1+J_2+J_3.            \] 
Firstly we have
\begin{align*}
J_1 &= \sum_{\ell_1 ,  \ell_2 \le x} a(\ell_1,\ell_2) \sum_{m_1 \le x / \ell_1 ,   m_2 \le x / \ell_2 } \bigl( b(m_1,m_2) -B \log m_1 \wedge m_2  \bigr)   \\
&= \sum_{\ell_1 ,  \ell_2 \le x} a(\ell_1,\ell_2)  \ o \Bigl( \frac{x}{\ell_1} \frac{x}{\ell_2} \log \frac{x}{\ell_1} \wedge \frac{x}{\ell_2} \Bigr)    
\end{align*}
Since $\log \frac{x}{k_1} \wedge \frac{x}{k_2} \le \log x$,  we have by Lemma 7 and Lemma 1
\[ J_1 \ll o \bigl( x^2 \log x \bigr) \sum_{\ell_1 ,  \ell_2 \le x} \frac{|a(\ell_1,\ell_2)|}{\ell_1 \ell_2}   = o \bigl( x^2 (\log x)^{\alpha+\beta+3} \bigr) .         \] 
Secondly we have by Lemma 5
\begin{align*}
J_2 &= B \sum_{m_1,  m_2 \le x} \log m_1 \wedge m_2 \sum_{\ell_1  \le x / m_1 , \ell_2 \le x / m_2} \bigl(a(\ell_1,\ell_2)-A(\log \ell_1)^\alpha (\log \ell_2)^{\beta} \bigr)   \\
&= B\sum_{m_1 , m_2 \le x} (\log m_1 \wedge m_2)  \ o \Bigl( \frac{x}{m_1} \frac{x}{m_2} \bigl(\log \frac{x}{m_1} \bigr)^\alpha \bigl(\log \frac{x}{m_2} \bigr)^{\beta} \Bigr)       \\
&= o \Bigl( x^2 (\log x)^{\alpha+\beta}  \sum_{m_1 , m_2 \le x} \frac{\log m_1 \wedge m_2}{m_1 m_2}\Bigr)= o \Bigl( x^2 (\log x)^{\alpha+\beta}  \cdot \frac{1}{3} (\log x)^3 \Bigr)  \\
&= o \bigl( x^2 (\log x)^{\alpha+\beta+3} \bigr) .        
\end{align*}
Thirdly we have by Lemma 4 and Lemma 6
\begin{align*}
J_3 &= AB \sum_{\ell_1 , \ell_2  \le x}   (\log \ell_1)^\alpha (\log \ell_2)^{\beta} \sum_{m_1  \le x / \ell_1 ,  m_2 \le x / \ell_2 } \log m_1 \wedge m_2    \\
&= AB \sum_{\ell_1 , \ell_2  \le x}   (\log \ell_1)^\alpha (\log \ell_2)^{\beta} \Bigl( \frac{x}{\ell_1}  \frac{x}{\ell_2}  \ \log \frac{x}{\ell_1} \wedge \frac{x}{\ell_2} + o( \frac{x}{\ell_1}  \frac{x}{\ell_2}  \ \log \frac{x}{\ell_1} \wedge \frac{x}{\ell_2}) \Bigr)         \\ 
&= AB \frac{x^2(\log x)^{\alpha+\beta+3}}{(\alpha+1)(\beta+1) (\alpha+\beta+3)}  + o \bigl(x^2(\log x)^{\alpha+\beta+3} \bigr) .     
\end{align*}
From these estimates we have 
\[\sum_{n_1,  n_2 \le x} c(n_1,n_2) =AB \frac{x^2(\log x)^{\alpha+\beta+3}}{(\alpha+1)(\beta+1) (\alpha+\beta+3)}  + o(x^2(\log x)^{\alpha+\beta+3} ) . \]
Thus the proof of Lemma 8 is now complete.
\end{proof}

Now we can prove Theorem 2. 
\begin{proof}[Proof of Theorem $2$] We first prove (i). We proceed by induction on $k$.
If $k=1$, then (\ref{eq:th2i}) holds by Theorem 1 in Ushiroya \cite{u}.
Let $k \ge 2$ and suppose that (\ref{eq:th2i}) holds for $k-1$ instead of $k$. 
We put $g=f * \mu_k$ and $h= g * \tau_{k-1}.$ Since
\[\sum_{n_1,n_2=1}^\infty \frac{|g(n_1,n_2)|}{n_1 n_2}=\sum_{n_1,n_2=1}^\infty \frac{|h*\mu_{k-1}(n_1,n_2)|}{n_1 n_2} < \infty \]
holds by the induction hypothesis, we obtain
\[ \lim_{x,y \to \infty} \frac{1}{xy (\log x \log y )^{k-2}  } \sum_{n_1 \le x  ,  n_2 \le y} h (n_1,n_2)  = C_{k-1} \sum_{n_1,  n_2=1}^{\infty} \frac{g(n_1,n_2)}{n_1 n_2}.      \] 
Since $f=h* \bold{1},$  we have by taking $a= h$, $b=\bold{1} $ and $\alpha=\beta=k-2$ in Lemma 8(i)
\[ \lim_{x,y \to \infty} \frac{1}{xy (\log x  \log y)^{k-1}  } \sum_{n_1 \le x  ,  n_2 \le y} f (n_1,n_2) \]
\[  = \frac{1}{(k-1)^2} C_{k-1} \sum_{n_1,  n_2=1}^{\infty} \frac{g(n_1,n_2)}{n_1 n_2}= C_k \sum_{n_1,  n_2=1}^{\infty} \frac{g(n_1,n_2)}{n_1 n_2}.          \] 
This proves (i).

Next we prove (ii). Similarly we proceed by induction on $k$. If $k=1$, then (\ref{eq:th2ii})  holds by Theorem 1. Let $k \ge 2$ and suppose that (\ref{eq:th2ii}) holds for $k-1$ instead of $k$. We put $g=f * \tilde{\mu}_k$ and $h= g * \tau_{k-1}.$  Since
\[ \sum_{n_1,n_2=1}^\infty \frac{|g(n_1,n_2)|}{n_1 n_2}=\sum_{n_1,n_2=1}^\infty \frac{|h*\mu_{k-1} (n_1,n_2)|}{n_1 n_2} < \infty, \]
we have by Theorem 2(i)
\[ \lim_{x,y \to \infty} \frac{1}{xy (\log x \log y )^{k-2}  } \sum_{n_1 \le x  ,  n_2 \le y} h (n_1,n_2)   =  C_{k-1} \sum_{n_1, \ n_2=1}^{\infty} \frac{g(n_1,n_2)}{n_1 n_2} .                \] 
Since $f=h* \tilde{\tau}_1,$ we have by taking $a= h$, $b=\tilde{\tau}_1 $ and $\alpha=\beta=k-2$ in Lemma 8(ii)
\[  \lim_{x \to \infty} \frac{1}{x^2 (\log x )^{2k-1}  } \sum_{n_1  ,  n_2 \le x} f (n_1,n_2) \]
\[=\frac{1}{(k-1)^2 (2k-1)} \frac{ C_{k-1}}{\zeta(2)} \sum_{n_1, \ n_2=1}^{\infty} \frac{g(n_1,n_2)}{n_1 n_2}   = \tilde{C}_k  \sum_{n_1, n_2=1}^{\infty} \frac{g(n_1,n_2)}{n_1 n_2}.            \] 
Thus the proof of Theorem 2 is now complete.
\end{proof}

\section{Multiplicative Case}

We say that $f$ is a multiplicative function of two variables if $f$ satisfies
\begin{equation}
\nonumber
f(m_1 n_1, m_2  n_2)=f(m_1 , m_2 )  \ f( n_1,n_2) 
\end{equation}
for any $m_1 ,m_2,  n_1, n_2  \in  \ \mathbb{N} $ satisfying $\gcd (m_1 m_2, \  n_1 n_2)=1 .$
It is well known that if $f$ and $g$ are multiplicative functions of two variables, then $f*g$ also becomes a multiplicative function of two variables. The next theorem is an extension of van der Corput's theorem (1.2) to the case in which $f$ is a multiplicative function of two variables.

\begin{theorem} Let $f$ be a multiplicative function of two variables and let $k \in \mathbb{N}$. \\
{\rm{(i)}} Suppose
\begin{equation}
\displaystyle  \sum_{p \in \mathcal{P}} \sum_{\substack{\nu_1 , \nu_2 \ge 0 \\ \nu_1 + \nu_2 \ge 1}} \frac{| (f*\mu_k) (p^{\nu_1} , p^{\nu_2} ) |}{p^{\nu_1+\nu_2}} < \infty .  \label{eq:th3(i)1} 
\end{equation}
Then we have  
\begin{equation}
\lim_{x,  y \to \infty} \frac{1}{x y (\log x \log y)^{k-1} } \sum_{\substack{n_1 \le x \\  n_2 \le y}} f(n_1,n_2)= C_k \prod_{p \in \mathcal{P}} \Bigl(1-\frac{1}{p} \Bigr)^{2k}  \Bigl(\sum_{ \nu_1,  \nu_2 \ge 0} \frac{ f (p^{\nu_1} , p^{\nu_2} )}{p^{\nu_1+\nu_2}} \Bigr),  \label{eq:th3(i)}              
\end{equation}
where $\displaystyle C_k=\frac{1}{((k-1)!)^2} $. \\ \\
{\rm{(ii)}} Suppose 
\begin{equation}
\sum_{p \in \mathcal{P}} \sum_{\substack{\nu_1 , \nu_2 \ge 0 \\ \nu_1 + \nu_2 \ge 1}} \frac{| (f* \tilde{\mu}_k) (p^{\nu_1} , p^{\nu_2} ) |}{p^{\nu_1+\nu_2}} < \infty .  \label{eq:th3(ii)1}  
\end{equation}
Then we have 
\begin{equation}
\lim_{x \to \infty} \frac{1}{x^2 (\log x)^{2k-1} } \sum_{n_1 \le x }  f(n_1,n_2)=  \tilde{C}_k' \prod_{p \in \mathcal{P}} \Bigl(1-\frac{1}{p} \Bigr)^{2k+1}  \Bigl(\sum_{ \nu_1,  \nu_2 \ge 0} \frac{ f (p^{\nu_1} , p^{\nu_2} )}{p^{\nu_1+\nu_2}} \Bigr), 
\end{equation}
where $\displaystyle \tilde{C}_k'= \zeta(2) \tilde{C}_k=\frac{1}{((k-1)!)^2 (2k-1)}. $
\end{theorem}

\begin{remark} \rm{In part} $\mathrm{(ii)}$, we do not deal with: \\
$ \lim_{x,y \to \infty} (xy (\log x  \log y)^{k-1} \log x \wedge y )^{-1} \sum_{n_1 \le x ,  n_2 \le y} f (n_1,n_2)$  since it is too complicated and we cannot obtain a simple formula.
\end{remark}

Before we prove Theorem 3, we give lemmas needed later.

\begin{lemma}[S$\acute{\rm{a}}$ndor and Crstici {\cite{sa}} p.107] For $k \in \mathbb{N}$ and $p \in \mathcal{P}$, we have
\[ \mu_k (p^{\nu_1}, p^{\nu_2}) = \left\{ \begin{array}{ll} (-1)^{\nu_1+\nu_2} \ \displaystyle{\binom{k}{\nu_1} \binom{k}{\nu_2} }  \ \ & if \ \ \nu_1, \ \nu_2 \le k, \vspace{0.1cm} \\ 0   &otherwise ,  \vspace{0.cm} \\  \end{array}  \right.      \] 
where $\binom{k}{\nu}$ is a binomial coefficient. 
\end{lemma}

\begin{lemma} For $p \in \mathcal{P}$ we have
\[\tilde{\mu} (p^{\nu_1},p^{\nu_2})=\left\{ \begin{array}{cl} -1 &  \ \ if \ \ \nu_1+\nu_2=1, \\ 2-p & \ \ if  \ \ \nu_1=\nu_2=1, \\ p-1  &  \ \ if \ \ |\nu_1-\nu_2|=1  \ \ and \ \ \nu_1,\nu_2 \ge 1,  \\ 2-2p  &  \ \ if \ \ \nu_1=\nu_2 \ge 2,  \\ 0 & \ \ otherwise. \\ \end{array}  \right.      \]
\end{lemma}
\begin{proof} Let $f$ be the multiplicative function defined by the same formulas as the above. Then, by an elementary calculation, it is easy to see that 
$(f * \gcd) (p^a,p^b)=\delta(p^a,p^b) $ holds for every $a,b \ge 0$. By the uniqueness of the Dirichlet inverse of the $\gcd$ function, we have $f=\tilde{\mu}$.
\end{proof}

Now we can prove Theorem 3.
\begin{proof}[Proof of Theorem 3] We first prove (i). 
Since the function: $ (n_1,n_2) \mapsto \frac{(f*\mu_k) (n_1,n_2)}{ n_1 n_2}$ is multiplicative, we have
\[ \sum_{n_1 \le x ,  n_2 \le y}  \frac{|(f*\mu_k) (n_1,n_2)|}{n_1 n_2} \le \prod_{p \in \mathcal{P}} \Bigl(\sum_{\nu_1, \ \nu_2 \ge 0} \frac{1}{p^{\nu_1+\nu_2}} |(f* \mu_k) (p^{\nu_1}, \ p^{\nu_2})|  \Bigr)       \]
\vspace{-0.5cm}
\begin{align*}
& =  \prod_{p \in \mathcal{P}} \Bigl(1+\sum_{\nu_1+\nu_2 \ge 1} \frac{1}{p^{\nu_1+\nu_2}} |(f* \mu_k) (p^{\nu_1}, \ p^{\nu_2})|  \Bigr)       \\
& \le  \exp \Bigl(\sum_p \Bigl(\sum_{\nu_1+ \nu_2 \ge 1} \frac{1}{p^{\nu_1+\nu_2}} |(f* \mu_k) (p^{\nu_1}, \ p^{\nu_2})|  \Bigr) \Bigr)  < \infty,        
\end{align*}
where we have used the well known inequality $1+x \le \exp (x)$ for $x\ge 0$. Therefore (\ref{eq:th2i}) holds by Theorem 2(i). On the other hand, using Lemma 9 we have
\[\sum_{ \nu_1,  \nu_2 \ge 0} \frac{ (f* \mu_k) (p^{\nu_1} , p^{\nu_2} ) }{p^{\nu_1+\nu_2}}=\sum_{ a_1,a_2,b_1, b_2 = 0}^{\infty} \frac{ f (p^{a_1} , p^{a_2} )  \ \mu_k (p^{b_1} , p^{b_2} ) }{p^{a_1+b_1+a_2+b_2}}  \]
\[= \sum_{ a_1,a_2= 0}^{\infty} \frac{ f (p^{a_1} , p^{a_2} )   }{p^{a_1+a_2}}  \sum_{ b_1,b_2= 0}^{k} \frac{ (-1)^{b_1+b_2} \ \displaystyle{\binom{k}{b_1} \binom{k}{b_2} }  }{p^{b_1+b_2}} = \sum_{ a_1,a_2= 0}^{\infty} \frac{ f (p^{a_1} , p^{a_2} )   }{p^{a_1+a_2}}  \Bigl(1-\frac{1}{p} \Bigr)^{2k}.           \] 
Hence the right side of (\ref{eq:th2i}) is equal to
\[C_k \prod_{p \in \mathcal{P}} \Bigl(1-\frac{1}{p} \Bigr)^{2k}  \Bigl(\sum_{ \nu_1,  \nu_2 \ge 0} \frac{ f (p^{\nu_1} , p^{\nu_2} )}{p^{\nu_1+\nu_2}} \Bigr). \]
This proves (i). 

Next we prove (ii). Similarly we have
\[ \sum_{m_1 ,m_2 \le x} \frac{|(f* \tilde{\mu}_k) (m_1,m_2)|}{m_1 m_2}  \le  \prod_{p \in \mathcal{P}} \Bigl(\sum_{\nu_1,  \nu_2 \ge 0} \frac{1}{p^{\nu_1+\nu_2}} |(f* \tilde{\mu}_k) (p^{\nu_1},  p^{\nu_2})|  \Bigr) \]
\vspace{-0.4cm}
\begin{align*}
& \le \prod_{p \in \mathcal{P}} \Bigl(1+\sum_{\nu_1+ \nu_2 \ge 1} \frac{1}{p^{\nu_1+\nu_2}} |(f* \tilde{\mu}_k) (p^{\nu_1},  p^{\nu_2})|  \Bigr)         \\
& \le \exp \Bigl(\sum_{p \in \mathcal{P}} \Bigl(\sum_{\nu_1+ \nu_2 \ge 1} \frac{1}{p^{\nu_1+\nu_2}} |(f* \tilde{\mu}_k) (p^{\nu_1},  p^{\nu_2})|  \Bigr) \Bigr)  < \infty.    
\end{align*}
Therefore (\ref{eq:th2ii}) holds by Theorem 2(ii). On the other hand, we have
\[\sum_{ \nu_1,  \nu_2 \ge 0} \frac{ (f* \tilde{\mu}_k) (p^{\nu_1} , p^{\nu_2} ) }{p^{\nu_1+\nu_2}} =\sum_{ a_1,a_2= 0}^{\infty} \frac{ f (p^{a_1} , p^{a_2} )   }{p^{a_1+a_2}} \sum_{ b_1,b_2= 0}^{\infty} \frac{  \tilde{\mu}_k (p^{b_1} , p^{b_2} ) }{p^{b_1+b_2}}.            \] 
If $k \ge 2$, then noting that $\tilde{\mu}_k=\mu_{k-1} * \tilde{\mu}$ we have 
\[\sum_{ b_1,b_2= 0}^{\infty} \frac{  \tilde{\mu}_{k} (p^{b_1} , p^{b_2} ) }{p^{b_1+b_2}} = \sum_{ c_1,c_2,d_1,d_2= 0}^{\infty} \frac{ \mu_{k-1} (p^{c_1} , p^{c_2} )   }{p^{c_1+c_2}}  \frac{  \tilde{\mu} (p^{d_1} , p^{d_2} ) }{p^{d_1+d_2}}  \]
\vspace{-0.4cm}
\begin{align*}
& =\sum_{c_1,c_2=0}^{k} \frac{(-1)^{c_1+c_2}}{p^{c_1+c_2}} \binom{k-1}{c_1} \binom{k-1}{c_2} \sum_{ d_1,d_2= 0}^{\infty} \frac{  \tilde{\mu} (p^{d_1} , p^{d_2} ) }{p^{d_1+d_2}}  \\
& = \Bigl(1-\frac{1}{p} \Bigr)^{2(k-1)} \sum_{ d_1,d_2= 0}^{\infty} \frac{  \tilde{\mu} (p^{d_1} , p^{d_2} ) }{p^{d_1+d_2}} .     
\end{align*}
Using the relation $\tilde{\mu} * \mathrm{gcd}= \delta $ we have
\[ \Bigl(\sum_{ d_1, d_2 = 0}^{\infty} \frac{  \tilde{\mu} (p^{d_1} , p^{d_2} ) }{p^{d_1+d_2}}\Bigr) \Bigl(\sum_{ d_1,  d_2 = 0}^{\infty} \frac{  \mathrm{gcd} (p^{d_1} , p^{d_2} ) }{p^{d_1+d_2}}  \Bigr) = 1,                  \] 
where, by an elementary calculation, we can easily derive
\[ \sum_{ d_1, d_2 = 0}^{\infty} \frac{  \mathrm{gcd} (p^{d_1} , p^{d_2} ) }{p^{d_1+d_2} } = \sum_{ d_1, d_2 = 0}^{\infty} \frac{  p^{d_1 \wedge d_2}  }{p^{d_1+d_2} } = \frac{1-\frac{1}{p^2}}{(1-\frac{1}{p})^3}.  \] 
Therefore we have obtained the following two formulas.
\begin{align}
\sum_{ b_1, b_2= 0}^{\infty} \frac{  \tilde{\mu} (p^{b_1} , p^{b_2} ) }{p^{b_1+b_2}} & =  \frac{\bigl(1-\frac{1}{p} \bigr)^3}{1-\frac{1}{p^2}}, \label{eq:sum tilde mu_k}   \\
\sum_{ b_1, b_2= 0}^{\infty} \frac{  \tilde{\mu}_k (p^{b_1} , p^{b_2} ) }{p^{b_1+b_2}} & = \Bigl(1-\frac{1}{p} \Bigr)^{2(k-1)}  \frac{\bigl(1-\frac{1}{p} \bigr)^3}{1-\frac{1}{p^2}}=  \frac{(1-\frac{1}{p} )^{2k+1} }{1-\frac{1}{p^2}}  \quad {\rm{if}} \quad k \ge 2 . \nonumber
\end{align}
Hence we see that, for every $k \in \mathbb{N}$, the right side of (\ref{eq:th2ii}) equals
\begin{align*}
\tilde{C}_k \prod_{p \in \mathcal{P}} \Bigl( \sum_{ \nu_1,  \nu_2 \ge 0} \frac{ (f* \tilde{\mu}_k) (p^{\nu_1} , p^{\nu_2} ) }{p^{\nu_1+\nu_2}}\Bigr)
 & =\tilde{C}_k \prod_{p \in \mathcal{P}} \Bigl(\sum_{ a_1,a_2= 0}^{\infty} \frac{ f (p^{a_1} , p^{a_2} )   }{p^{a_1+a_2}} \Bigr)  \frac{(1-\frac{1}{p} )^{2k+1} }{1-\frac{1}{p^2}}  \\   
& =\tilde{C}_k ' \prod_{p \in \mathcal{P}} \Bigl(1-\frac{1}{p} \Bigr)^{2k+1} \Bigl(\sum_{ \nu_1,\nu_2= 0}^{\infty} \frac{ f (p^{\nu_1} , p^{\nu_2} )   }{p^{\nu_1+\nu_2}} \Bigr) ,  
\end{align*} 
where $\tilde{C}_k '=\zeta(2) \tilde{C}_k$. 
Thus the proof of Theorem 3 is now complete.
\end{proof}

It is well known (Schwarz and Spilker \cite{sc}) that if $f: \mathbb{N} \mapsto \mathbb{C}$ is a multiplicative function satisfying 
$\sum_{p \in \mathcal{P}} ( |f(p)-1| /p + \sum_{\nu \ge 2} f(p^\nu) / p^\nu )<\infty,$ then the mean value $M(f)=\lim_{x \to \infty} x^{-1} \sum_{n \le x} f(n)$ exists and equals
$\prod_{p \in \mathcal{P}} (1-1/p ) (\sum_{ \nu  \ge 0}  f (p^{\nu} ) / p^{\nu} ) . $
The following theorem is a generalization of this result.

\begin{theorem} Let $f$ be a multiplicative function of two variables and let $k \in \mathbb{N}$. \\
{\rm{(i)}} Suppose
\begin{equation}
 \sum_{p \in \mathcal{P}} \Bigl( \frac{|f(p,1)-k|+|f(1,p)-k|}{p} + \sum_{ \nu_1 +\nu_2 \ge 2} \frac{|f(p^{\nu_1} , p^{\nu_2} )|}{p^{\nu_1+\nu_2}} \Bigr) < \infty .   \label{eq:th4i}          
\end{equation}
Then we have
\begin{equation}
\lim_{x,y \rightarrow \infty} \frac{1}{x y (\log x \log y)^{k-1} } \sum_{\substack{n_1 \le x \\  n_2 \le y}} f(n_1,n_2)=C_k  \prod_{p \in \mathcal{P}} \Bigl(1-\frac{1}{p} \Bigr)^{2k}  \Bigl(\sum_{ \nu_1,  \nu_2 \ge 0} \frac{ f (p^{\nu_1} , p^{\nu_2} )}{p^{\nu_1+\nu_2}} \Bigr) ,   \label{eq:th4(i)}
\end{equation}
where $\displaystyle C_k=\frac{1}{((k-1)!)^2}$. \\ \\
{\rm{(ii)}} Suppose
\begin{equation}
\sum_{p \in \mathcal{P}} \Bigl( \frac{|f(p,1)-k|+|f(1,p)-k|}{p}  + \frac{|f(p,p)-p|}{p^2} +\sum_{\substack{\nu_1+ \nu_2 \ge 2 \\ (\nu_1 ,\nu_2) \neq (1,1) }}  \frac{|f(p^{\nu_1} , p^{\nu_2} )|}{p^{\nu_1+\nu_2}}  \Bigr)< \infty .  \label{eq:th4(ii)1}  
\end{equation}
Then we have
\begin{equation}
\lim_{x \rightarrow \infty} \frac{1}{x^2 (\log x)^{2k-1} } \sum_{n_1 ,   n_2 \le x} f(n_1,n_2) = \tilde{C}_k ' \prod_{p \in \mathcal{P}} \Bigl(1-\frac{1}{p} \Bigr)^{2k+1}  \Bigl(\sum_{ \nu_1,  \nu_2 \ge 0} \frac{ f (p^{\nu_1} , p^{\nu_2} )}{p^{\nu_1+\nu_2}} \Bigr), \label{eq:th4(ii)}    
\end{equation}
where $\displaystyle \tilde{C}_k '=\frac{1}{((k-1)!)^2 (2k-1)} $.
\end{theorem}
\begin{proof}[Proof] We first prove (i). 
We would like to show that $f$ satisfies (\ref{eq:th3(i)1}). We have      
\[\sum_{p \in \mathcal{P}} \sum_{ \nu_1 + \nu_2 \ge 1} \frac{| (f*\mu_k) (p^{\nu_1} , p^{\nu_2} ) |}{p^{\nu_1+\nu_2}} =: I_1+I_2,                    \] 
where
\begin{align*}
I_1 &= \sum_{p \in \mathcal{P}} \sum_{ \nu_1 + \nu_2 = 1} \frac{| (f*\mu_k) (p^{\nu_1} , p^{\nu_2} ) |}{p^{\nu_1+\nu_2}} =\sum_{p \in \mathcal{P}} \frac{  |(f*\mu_k) (p,1)|+|(f*\mu_k)(1,p)|}{p}     \\
&= \sum_{p \in \mathcal{P}} \frac{ |f(p,1) -k|+|f(1,p) -k|}{p} < \infty , 
\end{align*}
and
\begin{align*}
I_2 &= \sum_{p \in \mathcal{P}} \sum_{ \nu_1 + \nu_2 \ge 2} \frac{| (f*\mu_k) (p^{\nu_1} , p^{\nu_2} ) |}{p^{\nu_1+\nu_2}} = \sum_{p \in \mathcal{P}} \sum_{ a_1 + a_2+b_1+b_2 \ge 2} \frac{| f (p^{a_1} , p^{a_2} ) \mu_k (p^{b_1} , p^{b_2} )  |}{p^{a_1+a_2+b_1+b_2}} .      \\
&= \sum_{p \in \mathcal{P}} \Biggl( \sum_{\substack{ a_1 + a_2=0 \\ b_1+b_2 \ge 2}} + \sum_{\substack{ a_1 + a_2=1 \\ b_1+b_2 \ge 1}} + \sum_{\substack{ a_1 + a_2 \ge 2 \\ b_1+b_2 \ge 0}} \Biggr) \frac{| f (p^{a_1} , p^{a_2} ) \mu_k (p^{b_1} , p^{b_2} )  |}{p^{a_1+a_2+b_1+b_2}}   \\
& \ll \sum_{p \in \mathcal{P}} \Biggl(\sum_{ b_1+b_2 \ge 2} \frac{1}{p^{b_1+b_2}}+\sum_{ b_1+b_2 \ge 1}\frac{|f(p,1)|+|f(1,p)|}{p^{1+b_1+b_2}}+\sum_{\substack{a_1+ a_2 \ge 2 \\ b_1 +b_2 \ge 0 }} \frac{| f (p^{a_1} , p^{a_2} )   |}{p^{a_1+a_2+b_1+b_2}}   \Biggr)  \\
& < \infty . 
\end{align*}
Therefore $f$ satisfies (\ref{eq:th3(i)1}), and hence (\ref{eq:th4(i)}) (which is equal to (\ref{eq:th3(i)})) holds by Theorem 3(i). This proves (i).

Next we prove (ii). If $k=1$, then it is easy to see that (\ref{eq:th4(ii)1}) implies (\ref{eq:th3(ii)1}) since $(f* \tilde{\mu})(p,1)=f(p,1)-1$,   \ $(f* \tilde{\mu})(1,p)=f(1,p)-1$  \ and \ $(f* \tilde{\mu})(p,p)=f(p,p)-f(p,1)-f(1,p)+2-p$ hold by Lemma 10. 
Let $k \ge 2$. We put $\tilde{f}=f*\tilde{\mu}$. We show that $\tilde{f}$ satisfies (\ref{eq:th4i}) for $k-1$ instead of $k$. We first see that            
\[\sum_{p \in \mathcal{P}} \frac{|\tilde{f}(p,1)-(k-1)|+|\tilde{f}(1,p)-(k-1)|}{p} =\sum_{p \in \mathcal{P}}  \frac{|f(p,1)-k|+|f(1,p)-k|}{p} < \infty. \]
We also have
\[\sum_{p \in \mathcal{P}}\sum_{ \nu_1 +\nu_2 \ge 2} \frac{|\tilde{f}(p^{\nu_1} , p^{\nu_2} )|}{p^{\nu_1+\nu_2}} =\sum_{p \in \mathcal{P}} \Bigl( \sum_{ \nu_1 +\nu_2 = 2} +\sum_{ \nu_1 +\nu_2 \ge 3} \Bigr) \frac{|\tilde{f}(p^{\nu_1} , p^{\nu_2} )|}{p^{\nu_1+\nu_2}} =: J_1+J_2,   \]
where
\[J_1=\sum_{p \in \mathcal{P}} \sum_{ \nu_1 +\nu_2 = 2}  \frac{|\tilde{f}(p^{\nu_1} , p^{\nu_2} )|}{p^{\nu_1+\nu_2}} =\sum_{p \in \mathcal{P}} \frac{|\tilde{f}(p^{2} , 1 )|+|\tilde{f}(p , p )|+|\tilde{f}(1 , p^2 )|}{p^2}.   \]
Noting that $\tilde{f}(p^{2} , 1 )=f(p^{2} , 1 )-f(p,1), \ \ \tilde{f}(p , p )=f(p,p)-f(p,1)-f(1,p)+2-p \ \ $ and $\tilde{f}(1,p^{2} )=f(1,p^{2}  )-f(1,p)$ hold by Lemma 10, we have
\[J_1 \ll \sum_{p \in \mathcal{P}} \frac{|f(p^2,1)|+|f(p,1)-k |+|f(p,p)-p|+|f(1,p) -k|+|f(1,p^2)|+1}{p^2}, \]
which implies that $J_1<\infty$.

As for $J_2$, since $ |\tilde{\mu}( p^{\nu_1},p^{\nu_2})| \ll 1+p $ holds for every $\nu_1, \nu_2 \ge 0$ by Lemma 10, we have
\[J_2=\sum_{p \in \mathcal{P}} \sum_{ \nu_1 +\nu_2 \ge 3}  \frac{|\tilde{f}(p^{\nu_1} , p^{\nu_2} )|}{p^{\nu_1+\nu_2}}  =\sum_{p \in \mathcal{P}} \sum_{ a_1+a_2+b_1+b_2 \ge 3}  \frac{|f (p^{a_1} , p^{a_2} ) \tilde{\mu} (p^{b_1},p^{b_2})| }{p^{a_1+a_2+b_1+b_2}}  \] 
\[  \ll \sum_{p \in \mathcal{P}} \Bigl(\sum_{ \nu_1 +\nu_2 \ge 2} \frac{1+|f(p^{\nu_1} , p^{\nu_2} )|}{p^{\nu_1+\nu_2}} \Bigr) <\infty.       \] 
Therefore $\tilde{f}$ satisfies (\ref{eq:th4i}) for $k-1$ instead of $k$. Hence by Theorem 4(i) we have     
\[ \lim_{x,y \to \infty} \frac{1}{x y (\log x \log y)^{k-2} } \sum_{\substack{n_1 \le x \\  n_2 \le y}} \tilde{f}(n_1,n_2)= C_{k-1} \prod_{p \in \mathcal{P}} \bigl(1-\frac{1}{p} \bigr)^{2(k-1)}  \bigl(\sum_{\substack{ \nu_1 \ge 0 \\ \nu_2 \ge 0}} \frac{ \tilde{f} (p^{\nu_1} , p^{\nu_2} )}{p^{\nu_1+\nu_2}} \bigr) .          \] 
Since $f=\tilde{f}*\tilde{\tau}_1,$ we have by taking $a= \tilde{f}$, $b=\tilde{\tau}_1$ and $\alpha=\beta=k-2$ in Lemma 8(ii)
\begin{align}
& \lim_{x \to \infty} \frac{1}{x^2 (\log x)^{2k-1} } \sum_{n_1,  n_2 \le x} f(n_1,n_2)  \nonumber \\
&=  \frac{1}{(k-1)^2 (2k-1)} \frac{1}{\zeta(2)} C_{k-1} \prod_{p \in \mathcal{P}} \Bigl(1-\frac{1}{p} \Bigr)^{2(k-1)} \Bigl(\sum_{ \nu_1, \ \nu_2 \ge 0} \frac{ \tilde{f} (p^{\nu_1} , p^{\nu_2} )}{p^{\nu_1+\nu_2}} \Bigr)   \nonumber \\
&= \frac{1}{\zeta(2)}  \tilde{C}_{k} ' \prod_{p \in \mathcal{P}} \Bigl(1-\frac{1}{p} \Bigr)^{2(k-1)}  \Bigl(\sum_{ a_1,a_2,b_1,b_2  \ge 0} \frac{ f (p^{a_1} , p^{a_2} )}{p^{a_1+a_2}} \frac{ \tilde{\mu} (p^{b_1} , p^{b_2} )}{p^{b_1+b_2}}  \Bigr) .  \nonumber   
\end{align}
By (\ref{eq:sum tilde mu_k}) we see that the above equals
\begin{align*}
 \frac{1}{\zeta(2)} & \tilde{C}_{k} ' \prod_{p \in \mathcal{P}} \Bigl(1-\frac{1}{p} \Bigr)^{2(k-1)}  \frac{(1-\frac{1}{p} )^3}{1-\frac{1}{p^2}} \Bigl(\sum_{ a_1,a_2  \ge 0} \frac{ f (p^{a_1} , p^{a_2} )}{p^{a_1+a_2}}  \Bigr)   \\
= & \tilde{C}_{k} ' \prod_{p \in \mathcal{P}} \Bigl(1-\frac{1}{p} \Bigr)^{2k+1}  \Bigl(\sum_{ \nu_1, \nu_2  \ge 0} \frac{ f (p^{\nu_1} , p^{\nu_2} )}{p^{\nu_1+\nu_2}}  \Bigr)  .  
\end{align*}
Thus the proof of Theorem 4 is now complete.
\end{proof}

\section{Examples} 
Let $\omega (n)=\sum_{p|n} 1$ be the counting function of the total number of prime factors of $n$ taken without multiplicity. It is known that for a fixed positive integer $k$, $\lim_{x\to \infty} x^{-1} (\log x)^{1-k} \sum_{n \le x} k^{\omega (n)}=((k-1)!)^{-1} \prod_{p \in \mathcal{P}} (1-1/p )^{k-1} \Bigl(1+(k-1)/p \Bigr)$ (cf. Tenenbaum and Wu \cite{te} p.25). 
The following example is an extenstion of this result to the case of a function of two variables.

\begin{example} Let  $k \in \mathbb{N}$ and let $f(n_1, n_2 )=k^{\omega (n_1n_2)} $. Then we have 
\[\lim_{x,y \to \infty} \frac{1}{x y (\log x \log y)^{k-1} } \sum_{\substack{n_1 \le x \\  n_2 \le y}} f(n_1,n_2)= C_k \prod_{p \in \mathcal{P}} \Bigl(1-\frac{1}{p} \Bigr)^{2(k-1)} \Bigl(1+\frac{2(k-1)}{p}+\frac{1-k}{p^2} \Bigr),   \]
where $\displaystyle C_k=\frac{1}{((k-1)!)^2}$.
\end{example}
\begin{proof} Since $f(p^{\nu_1},p^{\nu_2})=k$ if $\nu_1+\nu_2 \ge 1 $, it is easy to see that $f$ satisfies (\ref{eq:th4i}). Therefore we can apply Theorem 4(i) to obtain
\[\lim_{x,y \to \infty} \frac{1}{x y (\log x \log y)^{k-1} } \sum_{n_1 \le x ,  n_2 \le y} f(n_1,n_2)= C_k  \prod_{p \in \mathcal{P}} \Bigl(1-\frac{1}{p} \Bigr)^{2k}  \Bigl(1+\sum_{ \nu_1+  \nu_2 \ge 1} \frac{ k}{p^{\nu_1+\nu_2}} \Bigr) \]
\[= C_k \prod_{p \in \mathcal{P}} \Bigl(1-\frac{1}{p} \Bigr)^{2k}  \Bigl(1+ \frac{k(2p-1)}{(p-1)^2} \Bigr)  \\
= C_k \prod_{p \in \mathcal{P}} \Bigl(1-\frac{1}{p} \Bigr)^{2(k-1)} \Bigl(1+\frac{2(k-1)}{p}+\frac{1-k}{p^2} \Bigr).  \]
\end{proof}

\begin{example} Let $f(q , n )=|c_q (n)|$ where $c_q (n)=\mu (q/(q,n)) \varphi (q) / \varphi ( q/(q,n) ) $ is the Ramanujan sum. Then we have 
\[\lim_{x \to \infty} \frac{1}{x^2 \log x } \sum_{n_1, n_2 \le x} f(n_1,n_2)=\prod_{p \in \mathcal{P}} \Bigl(1-\frac{3}{p^2}+\frac{2}{p^3} \Bigr).   \]
\end{example}
\begin{proof} It is easy to see that 
$f(p,1)=f(1,p)=1, \ \ f(p,p)=p-1, \\ f(p^\nu,1)=0$, \ $f (1,p^\nu)=1$ \  if \  $ \nu \ge 2, $ and
\[f(p^{\nu_1},p^{\nu_2}) = \left\{ \begin{array}{ll} \mu^2 (p^{\nu_1-\nu_2}) p^{\nu_2} &  \ \ \mathrm{if} \quad 1 \le \nu_2 <\nu_1,     \vspace{0.1cm} \\ p^{\nu_1} (1-1/p)  &  \ \ \mathrm{if} \quad 1 \le \nu_1 \le \nu_2 .     \vspace{0.cm}  \\  \end{array}  \right.  \]
From these relations, we see that $f$ satisfies (\ref{eq:th4(ii)1}) for $k=1$. After an elementary calculation we obtain
\[\sum_{ \nu_1,  \nu_2 \ge 0} \frac{ f(p^{\nu_1},p^{\nu_2})}{p^{\nu_1+\nu_2}}=\frac{p+2}{p-1} . \]
Therefore we have by (\ref{eq:th4(ii)})
\[\lim_{x \to \infty} \frac{1}{x^2 \log x \ } \sum_{n_1,  n_2 \le x} f(n_1,n_2) =\tilde{C}_1 ' \prod_{p \in \mathcal{P}} \Bigl(1-\frac{1}{p} \Bigr)^{3} \frac{p+2}{p-1} =\prod_{p \in \mathcal{P}} \Bigl(1-\frac{3}{p^2}+\frac{2}{p^3} \Bigr) . \]
\end{proof}

Next we obtain the leading coefficients in (1.3) and (1.4) using Theorem 4.

\begin{example} Let  $f(n_1, n_2 )=\sigma( \gcd (n_1,n_2)) $ where $\sigma (n)= \sum_{d|n} d$. Then we have 
\[\lim_{x \to \infty} \frac{1}{x^2 \log x } \sum_{n_1, n_2 \le x} f(n_1,n_2)=1 . \]
\end{example}
\begin{proof} Since $ f(p^{\nu_1},p^{\nu_2})=(p^{\nu_1 \wedge \nu_2 +1}-1)/(p-1) $ if $\nu_1,\nu_2 \ge 0$, it is easy to see that $f$ satisfies (\ref{eq:th4(ii)1}) for $k=1$. Therefore we can apply Theorem 4(ii) for $k=1$.
After an elementary calculation we obtain
\[\sum_{ \nu_1,  \nu_2 \ge 0} \frac{ f(p^{\nu_1},p^{\nu_2})}{p^{\nu_1+\nu_2}}=\frac{1}{(1-\frac{1}{p})^3} . \]
Therefore  we have by (\ref{eq:th4(ii)})
\[\lim_{x \to \infty} \frac{1}{x^2 \log x \ } \sum_{n_1,  n_2 \le x} f(n_1,n_2) =\tilde{C}_1 ' \prod_{p \in \mathcal{P}} \Bigl(1-\frac{1}{p} \Bigr)^{3}\frac{1}{(1-\frac{1}{p})^3}  =1 . \]
\end{proof}

\begin{example} Let  $f(n_1, n_2 )=\varphi( \gcd (n_1,n_2)) $. Then we have 
\[\lim_{x \to \infty} \frac{1}{x^2 \log x } \sum_{n_1, n_2 \le x} f(n_1,n_2)=\frac{1}{\zeta^2 (2)} . \]
\end{example}
\begin{proof} Since $f(p^{\nu_1},p^{\nu_2})=p^{\nu_1 \wedge \nu_2 } (1-1/p ) $ if  $\nu_1, \nu_2 \ge 1$, it is easy to see that $f$ satisfies (\ref{eq:th4(ii)1}) for $k=1$. Therefore we can apply Theorem 4 (ii) for $k=1$.
After an elementary calculation we obtain
\[\sum_{ \nu_1,  \nu_2 \ge 0} \frac{ f(p^{\nu_1},p^{\nu_2})}{p^{\nu_1+\nu_2}}=\frac{(1+\frac{1}{p})^2}{1-\frac{1}{p}} . \]
Therefore we have by (\ref{eq:th4(ii)})
\[\lim_{x \to \infty} \frac{1}{x^2 \log x \ } \sum_{n_1,  n_2 \le x} f(n_1,n_2) =\prod_{p \in \mathcal{P}} \Bigl(1-\frac{1}{p} \Bigr)^{3} \frac{(1+\frac{1}{p})^2}{1-\frac{1}{p}} =\prod_{p \in \mathcal{P}} \Bigl(1-\frac{1}{p^2} \Bigr)^{2}   = \frac{1}{\zeta^2 (2)}. \]
\end{proof}

The proof of the following example is similar.

\begin{example} Let 
\begin{align*}
& f_1(n_1, n_2 )=\gcd (n_1,n_2) \mu^2 (\gcd (n_1,n_2)) , \\
& f_2 (n_1, n_2 )=\gcd (n_1,n_2) \mu^2 (\mathrm{lcm} (n_1,n_2)) . 
\end{align*}
Then we have 
\begin{align*}
& \lim_{x \to \infty} \frac{1}{x^2 \log x } \sum_{n_1, n_2 \le x} f_1(n_1,n_2)=\frac{1}{\zeta^2 (2)},  \\
& \lim_{x \to \infty} \frac{1}{x^2 \log x } \sum_{n_1, n_2 \le x} f_2(n_1,n_2)=\prod_{p \in \mathcal{P}} \Bigl(1-\frac{1}{p} \Bigr)^3 \Bigl(1+\frac{3}{p} \Bigr) .  
\end{align*}
\end{example}

\begin{example} Let $\displaystyle f(n_1, n_2 )=  \frac{\phi (n_1) \phi (n_2)}{\mathrm{lcm} (n_1,n_2)}  $. Then we have 
\[\lim_{x \to \infty} \frac{1}{x^2 \log x  } \sum_{n_1, n_2 \le x} f(n_1,n_2)=\prod_{p \in \mathcal{P}} \Bigl(1-\frac{1}{p} \Bigr)^3 \Bigl(1+\frac{3}{p}+\frac{1}{p^2} \Bigr)  . \]

\end{example}
\begin{proof} Since $f(p^{\nu},1)=f(1,p^{\nu})=1-1/p $ \ if \ $\nu \ge 1$ and \\ $ f(p^{\nu_1},p^{\nu_2})=(1-1/p )^2 p^{\nu_1 \wedge \nu_2 } $ \ if \ $\nu_1, \nu_2 \ge 1$, it is easy to see that $f$ satisfies (\ref{eq:th4(ii)1}) for $k=1$. Therefore we can apply Theorem 4 (ii) for $k=1$.
After an elementary calculation we obtain
\[\sum_{ \nu_1,  \nu_2 \ge 0} \frac{ f(p^{\nu_1},p^{\nu_2})}{p^{\nu_1+\nu_2}}=1+\frac{3}{p}+\frac{1}{p^2} . \]
Therefore, using (\ref{eq:th4(ii)}) for $k=1$, we have the desired result.
\end{proof}

Next we obtain the leading coefficients in (1.5) and (1.6).

\begin{example} Let  $ s(n_1,n_2)= \sum_{d_1 | n_1, d_2 | n_2}  \gcd(d_1,d_2) $. Then we have 
\[\lim_{x \to \infty} \frac{1}{x^2 (\log x)^3 } \sum_{n_1, n_2 \le x} s(n_1,n_2)=\frac{2}{\pi^2}. \]
\end{example}
\begin{proof} Since $s=gcd* \bold{1}=\tilde{\tau}_2,$ we have $s* \tilde{\mu}_2 =\delta.$ Therefore (\ref{eq:th2(ii)1}) trivially holds for $k=2$ and (\ref{eq:th2ii}) gives
\[\lim_{x \to \infty} \frac{1}{x^2 (\log x)^3 } \sum_{n_1, n_2 \le x} s(n_1,n_2)=\tilde{C}_2 ' \sum_{n_1, n_2 \le x} \frac{\delta(n_1,n_2)}{n_1 n_2}=\frac{2}{\pi^2}.   \]
\end{proof}

\begin{example} Let  $ c(n_1,n_2)= \sum_{d_1 | n_1, d_2 | n_2}  \varphi (\gcd(d_1,d_2)) $. Then we have
\[\lim_{x \to \infty} \frac{1}{x^2 (\log x)^3 } \sum_{n_1, n_2 \le x} c(n_1,n_2)=\frac{12}{\pi^4}. \]
\end{example}

\begin{proof} we note that $c=\varphi(\gcd)* \bold{1}$.  Since $\varphi (\gcd)$ satisfies (\ref{eq:th4(ii)1}) for $k=1$ from the proof of Example 4, we see that $\varphi (\gcd)$ also satisfies (2.1) from the proofs of Theorem 4, Theorem 3 and Theorem 2. Therefore we have by Theorem 1 and Example 4
\[\lim_{x,y \to \infty} \frac{1}{xy  \log x \wedge y} \sum_{n_1 \le x ,  n_2 \le y}  \varphi(\gcd (n_1,n_2)) =\frac{1}{\zeta^2 (2)}  . \] 
Taking $a=\bold{1}, \ b=\varphi(\gcd)$ and $\alpha=\beta=0$ in Lemma 8(ii), we have
\[\lim_{x \to \infty} \frac{1}{x^2 (\log x)^3 } \sum_{n_1, n_2 \le x} c(n_1,n_2)=\frac{1}{3}\frac{1}{\zeta^2 (2)}=\frac{12}{\pi^4}.    \]
\end{proof}

\begin{remark} \rm{According to Novak and T$\acute{\rm{o}}$th} \cite{no}, \rm{it holds that} $c(p,1)=c(1,p)=2$, $c(p,p)=p+2$, 
$c(p^a,1)=c(1,p^a)=a+1$ \ if \ $a \ge 1$, \ and, moreover, \\
$c(p^a,p^b)=2(1+p+p^2+ \ldots + p^{a-1})+(b-a+1)p^a \quad \mathrm{if} \quad 1 \le a \le b . $
Using this explicit formulas we can directly show that $c$ satisfies (\ref{eq:th4(ii)1}) for $k=2$ and also can directly calculate (\ref{eq:th4(ii)}). However, we did not prove in that way for simplicity. 
\end{remark}

\begin{example}
Let $ A(n_1, n_2 )=  \sum_{d_1 | n_1,  \ d_2 | n_2}  \phi (d_1) \phi (d_2) /\mathrm{lcm} (d_1,d_2) . $
Then we have 
\begin{equation}
\lim_{x \to \infty} \frac{1}{x^2 (\log x)^3 } \sum_{n_1, n_2 \le x} A(n_1,n_2)=\frac{1}{3} \prod_{p \in \mathcal{P}} \Bigl(1-\frac{1}{p} \Bigr)^3 \Bigl(1+\frac{3}{p}+\frac{1}{p^2} \Bigr). \label{eq:ex9} 
\end{equation}
\end{example}
\begin{proof} Let $g(n_1,n_2)=\phi (n_1) \phi (n_2) / \mathrm{lcm} (n_1,n_2)$. Since $A=g* \bold{1},$ by a similar argument as in Example 8, we see that the left side of (\ref{eq:ex9}) equals
$$ \frac{1}{3} \lim_{x \to \infty} \frac{1}{x^2 \log x } \sum_{n_1, n_2 \le x} g(n_1,n_2) . $$
By Example 6, it is easy to see that the above equals the right side of (\ref{eq:ex9}).
\end{proof}


\end{document}